\documentclass[english]{article}
\usepackage[utf8]{inputenc}
\usepackage[T1]{fontenc}
\usepackage{amsmath,amsthm,amssymb,mathtools}
\usepackage{graphicx}
\usepackage{multicol}
\usepackage{hyperref}

\usepackage{todonotes}
\usepackage{color}
\usepackage{algorithm,algorithmic}

\def\Z{\mathbb Z}
\def\Q{\mathbb Q}
\def\Qbar{\overline{\mathbb{Q}}}
\def\R{\mathbb R}
\def\C{\mathbb C}

\def\Oo{\mathcal O}
\def\Hh{\mathcal H}

\DeclareMathOperator \Log {Log}
\DeclareMathOperator \Reg {Reg}
\DeclareMathOperator \Vol {Vol}

\DeclarePairedDelimiter\abs{\lvert}{\rvert}
\DeclarePairedDelimiter\norm{\lVert}{\rVert}

\newtheorem{lem}{Lemma}
\newtheorem{thm}[lem]{Theorem}
\newtheorem{prop}[lem]{Proposition}

\newtheorem{de}[lem]{Definition}

\newtheorem{rmk}[lem]{Remark}

\begin{document}

\title{Pisot unit generators in number fields}
\author{T.~V\'avra, F.~Veneziano}

\maketitle

\begin{abstract}
%
Pisot numbers are real algebraic integers bigger than 1, whose other conjugates have all modulus smaller than 1. In this paper we deal with the algorithmic problem of finding the smallest Pisot unit generating a given number field.
We first solve this problem in all real fields, then we consider the analogous problem involving the so called complex Pisot numbers and we solve it in all number fields that admit such a generator, in particular all fields without CM, but not only those.
\end{abstract}

\section{Introduction}
  Pisot numbers are a remarkable class of algebraic numbers having the following definition.
      \begin{de}
         A Pisot number is a real algebraic integer greater than 1 whose other conjugates have modulus strictly smaller than 1.
      \end{de}
      Pisot numbers were introduced by Pisot in his thesis in 1938, although they had been considered earlier by Thue and Hardy. Pisot was mainly concerned with the link to harmonic analysis, but by virtue of their arithmetic properties they arise naturally in many other fields like ergodic theory, dynamical systems, algebraic groups, and non-standard numeration.
      
%
%
%

       The first aim of this paper is to present an algorithm that, given a real number field, finds a generator over $\Q$ which is a unit and a Pisot number. As a motivation for this purpose we recall some results from the theory of number systems.

Consider $\beta$-expansions, a number system where elements $x\in[0,1)$ are expressed in the form $x=\sum_{i\geq 1}a_i\beta^{-i}$ with $a_i\in\{0,1,\dots, \lceil\beta\rceil-1\}$
defined as
$$
a_i=\lfloor\beta T_\beta^{i-1}(x)\rfloor,\quad\text{ where }\quad T_\beta: [0,1)\mapsto[0,1),\ T_\beta(x)=\beta x-\lfloor\beta x\rfloor.
$$

      A well known theorem of K. Schmidt \cite{Schmidt} states that if $\beta$ is a Pisot number, then all the elements of $\Q(\beta)\cap[0,1)$ have eventually periodic $\beta$-expansion. In other words, the orbit of $x\in[0,1)$ under $$T_\beta:x\mapsto \beta x\mod 1$$ is eventually periodic if and only if $x\in\Q(\beta).$

Even more can be said under the additional assumption that $\beta$ is a unit of $\Q(\beta).$
S. Akiyama showed in \cite{Aki98} that if $\beta$ is a Pisot unit which satisfies an additional technical condition, then there exists a real number $c(\beta)>0$ such that every element of $\Q\cap[0,c(\beta))$ has purely periodic expansion.
     
It was already proved by Salem in \cite{salem} that in every number field $K$ there is a Pisot number $\beta$ such that $K=\Q(\beta)$. His proof is based on the existence of a lattice point in a certain volume in $\R^n$; finding such a point, however, is algorithmically a hard problem. Q. Cheng and J. Zhuang in \cite{ChengZhuang} presented a polynomial time algorithm to find a Pisot generator of a real Galois field when its integral basis is given; however they assume that the field is a Galois, while a generator exists for every real number field.

\medskip
      
      It is also known that every real number field can be generated by a Pisot unit, see Teorem 5.2.2 of \cite{PisotSalemNumbers} for a non-constructive proof.

      Our Algorithm \ref{Algo2} solves this problem computationally.
      The idea is to use the map $\Log:\Oo_K\setminus\{0\}\to\R^{r_1+r_2}$ as in the proof of the famous theorem by Dirichlet on the structure of $\Oo_K^*$ to reduce the problem to a lattice problem in $\R^{r_1+r_2}$. This is done by observing that all Pisot units lie inside an open convex subset $Q$ of $\R^{r_1+r_2}$ bounded by planes, and belong to a certain lattice $\Lambda$.
      
      We are interested in smallest Pisot generators of $K$, and for this purpose we utilize Integer Programming algorithms. Note that linear programming works on a closed convex set, while we look for solutions in an open set, and it may happen that potential solutions exist on the boundary of $Q$ (they correspond to Salem numbers, another important class of algebraic integers with remarkable arithmetic properties). To deal with this problem, we use an explicit result on the Lehmer Conjecture to add an additional linear constraint that excludes all lattice points on the boundary and none of those lying in the interior.
   
\medskip

      Clearly only real number fields can be generated by Pisot numbers, but it is evident from the definition of the $\Log$ map that real and complex embeddings of the field $K$ play a very similar role in the proof, and that an analogous theorem can be obtained also for complex fields if one changes the definitions accordingly.
      
      A natural extension of the definition of Pisot numbers to non-real algebraic integers is the following
      \begin{de}
	 An algebraic integer  $\theta\in\Qbar\setminus\R$ is called a complex Pisot number if $\abs{\theta}>1$ and all its conjugates except $\overline{\theta}$ have modulus strictly smaller than 1.
      \end{de}
      This definition has been considered by several mathematicians in the field of numeration systems as a suitable analogue of Pisot numbers.

      Accordingly, a second aim of this paper is to present an algorithm for determining a complex Pisot unit generator of a given complex field. It turns out that not all complex number fields can be generated by complex Pisot units (examples in which this is impossible are the fields $\Q(i)$ and $\Q(\sqrt{2},\sqrt{-5})$), but we are able to give a characterisation of the fields for which this is possible and to exhibit Algorithm \ref{Algo3}, which finds such a generator whenever it exists and halts with an error message if it doesn't exist.

      Algorithm \ref{Algo3} is based on the same idea of algorithms \ref{Algo1} and \ref{Algo2}. The main difference is that in this case not all points in $\Lambda\cap Q$ correspond to the units that we seek, and some adjustments are needed to exclude from the search the points lying in a certain sublattice.

We state here the problems that we have considered and solved in this paper. Following \cite{PisotSalemNumbers} we also define the set of the $U$-numbers are the union of the set of Pisot and Salem numbers (the precise definition is found in Section~\ref{preliminaries}).

\begin{enumerate}
 \item[P1] Given a real number field $K$, find the smallest $U$-number $\alpha\in\Oo_K^*$ such that $K=\Q(\alpha)$.
 \item[P2] Given a real number field $K$, find the smallest Pisot number $\alpha\in\Oo_K^*$ such that $K=\Q(\alpha)$.
 \item[P3] \label{prob:cp} Given a complex number field $K$, decide whether a complex Pisot number $\alpha\in\Oo_K^*$ exists, such that $K=\Q(\alpha)$. If any exist, find the smallest in absolute value.
\end{enumerate}
Our three algorithms {\tt FINDMIN}, {\tt CUTEDGE}, {\tt FINDCPISOT} solve these three problems. More precisely, we prove in Section~\ref{mainsection} the following theorems:
\begin{thm}\label{thm:algo1}
   Algorithm {\tt FINDMIN} always terminates with output $\alpha$. When $K$ is a real field, and $\phi_k$ is the identity embedding, then $\abs{\alpha}$ is a $U$-number, an algebraic unit, it generates the field $K$ over $\Q$, and it has the smallest height among elements with this properties.
\end{thm}

 \begin{thm}\label{thm:algo2}
  The Algorithm {\tt CUTEDGE} always terminates with output $\alpha$. When $K$ is a real field, and $\phi_k$ is the identity embedding, then $\abs{\alpha}$ is a Pisot number, an algebraic unit, it generates the field $K$ over $\Q$, and it has the smallest height among elements with this properties.
 \end{thm}

 \begin{thm}\label{thm:algo3}
    The Algorithm {\tt FINDCPISOT} always terminates. It returns an error message when $\Oo_K^*\subseteq\R$; in all other cases is outputs an element $\beta$ such that: $\beta$ is a complex Pisot number, an algebraic unit, it generates the field $K$ over $\Q$, and it has the smallest height among elements with this properties.
 \end{thm}
 
 We have also written an implementation of our Algorithms 1 and 2 in SageMath. The source code is available at \url{https://cloud.sagemath.com/projects/963c4563-d55c-4db1-9d1f-e4401cac00d0/files/pisotgen.sagews}.

%
%

\subsubsection*{Content of the paper}
The paper is organized as follows.
In Section \ref{subsec:nf} we introduce the notation that will be used in the whole paper; we describe the structure of the units of a number field and the map $\Log$ that is used to translate our problems into Integer Programming problems; we also define CM fields, as they appear in the characterisation of the fields for which problem \ref{prob:cp} cannot be solved. 

In Section \ref{subsec:salem} we give the definition of Salem numbers and some of their properties, which will be used in the proof of Theorem \ref{thm:algo2}.

In Section \ref{subsec:height} we define the height function and we introduce the Lehmer Conjecture, an open problem in Number Theory which has connections to Salem numbers.

Section~\ref{mainsection} contains the main results of this paper. We explain how Pisot number and units behave under the $\Log$ map, we present our three algorithms and we prove that they solve the problems P1,P2,P3 stated in the introduction.

We conclude the paper by showing, in Section~\ref{subsec:bound}, how to give a simple bound for the output of Algorithms \ref{Algo1}, \ref{Algo2} in terms of the input.

\section{Notation and preliminaries}\label{preliminaries}
In this section we recall the basic definitions and notations in algebraic number theory that we will need for our theorem. We also give the definition of Salem and Pisot numbers and a natural generalization to the complex plane, and recall some of their known properties. 
   \subsection{Number fields}\label{subsec:nf}
      \subsubsection*{Notation}
	 We will indicate by $K$ a number field and by $\Oo_K$ its ring of integers; as customary, $K^*$ and $\Oo_K^*$ will be the multiplicative groups of invertible elements in $K$ and $\Oo_K$ respectively.
	 For all complex numbers $z$, a horizontal bar will denote the usual complex conjugation.
	 We denote by $\sigma_1,\dotsc,\sigma_{r_1}$ the embeddings of $K$ in $\R$, and by $\tau_1,\overline{\tau_1},\dotsc,\tau_{r_2},\overline{\tau_{r_2}}$ the remaining embeddings of $K$ in $\C$. (It is clear that the non-real embeddings appear in complex conjugate pairs.) Let $n$ be the dimension of $K$ over $\Q$, so that $n=r_1+2r_2$.
	 When there is no need to distinguish real and complex embeddings, we write $(\phi_1,\dotsc,\phi_n)=(\sigma_1,\dotsc,\sigma_{r_1},\tau_1,\dotsc,\tau_{r_2},\overline{\tau_{r_1}},\dotsc,\overline{\tau_{r_2}})$.
	 Let also $\mu_K$ indicate the cyclic group of all the roots of unity belonging to $K$.
   
%
%

%
      \subsubsection*{The structure of the units}\label{structure}
	The structure of the group $\Oo_K^*$ is described by the classical theorem of Dirichlet, which states that $\Oo_K^*\cong \mu_K\times\Z^{r_1+r_2-1}$. In order to prove this result one introduces the following mapping:
	 \begin{align*}
	    \Log:\Oo_K\setminus\{0\} &\to \R^{r_1+r_2}\\	
	    x &\mapsto \left(\log\abs{\sigma_1(x)},\dotsc,\log\abs{\sigma_{r_1}(x)},2\log\abs{\tau_1(x)},\dotsc,2\log\abs{\tau_{r_2}(x)}\right).
	 \end{align*}
	 It is easy to see that this map is a group homomorphism from the multiplicative structure of $\Oo_K$ to the additive structure of $\R^{r_1+r_2}$.
	 Furthermore $\Log(\cdot)$ maps $\Oo_K^*$ to a lattice $\Lambda$ of rank $r_1+r_2-1$ contained in the hyperplane $\Hh\subset\R^{r_1+r_2}$ defined by the equation $x_1+\dotsb+x_{r_1+r_2}=0$.
	 
	 The kernel of $\Log$ is precisely the set $\mu_K$ of the roots of unity contained in $K$; hence one proves that $\Oo_K^*\cong \mu_K\times\Lambda\cong \mu_K\times\Z^{r_1+r_2-1}$. We will write $r=r_1+r_2-1$. (The details can be found in any book on classical algebraic number theory; for example \cite{marcus})
	 
	 A set of $r$ units in $\Oo_K^*$ whose images through $\Log$ form a basis of $\Lambda$ is called a fundamental system of units in $\Oo_K$.
	 
	We remark that a fundamental system of units can be computed algorithmically, see \cite{cohen}.
	 \subsubsection*{The regulator}	 
	 Fix a fundamental system of units $\{u_1,\dotsc,u_r\}$, and consider the $(r+1)\times r$ matrix
	 \begin{equation}\label{LabelMatrixA}A=\left(
	    \begin{array}{ccc}
	       \log\abs{\sigma_1(u_1)}	& \dots		& \log\abs{\sigma_1(u_r)}	\\
	       \vdots			& \ddots	& \vdots			\\
	       \log\abs{\sigma_{r_1}(u_1)} & \dots		& \log\abs{\sigma_{r_1}(u_r)}	\\
	       2\log\abs{\tau_1(u_1)} 	& \dots		& 2\log\abs{\tau_1(u_r)}	\\
	       \vdots			& \ddots	& \vdots			\\
	       2\log\abs{\tau_{r_2}(u_1)}	& \dots 	& 2\log\abs{\tau_{r_2}(u_r)}	
	    \end{array}\right)
	 \end{equation}   
	 whose columns are the images of the fundamental system through $\Log$.
	 
	 Let us define the regulator $\Reg (\Oo_K)$ as the absolute value of the determinant of the square matrix obtained by deleting a row from $A$. Because the sum of the entries in every column of $A$ is $0$, one can prove that this quantity is well-defined.
	 Furthermore, one can see that $\Reg (\Oo_K)= \frac{\Vol(\Lambda)}{\sqrt{r+1}}$ where $\Vol(\Lambda)$ is the ($r$-dimensional) volume of a fundamental domain of $\Lambda$, and therefore $\Reg (\Oo_K)$ does not depend on the choice of a particular fundamental system.
	 
	 The regulator of $\Oo_K$ is an important invariant of the field $K$ and measures the complexity of the units. It is closely related to the class number and appears in many important formulae which describe the distribution of the ideals of $\Oo_K$.
      
         \subsubsection*{CM fields}
          A remarkable class or number fields with several interesting algebraic properties are the so called CM fields.
	 \begin{de}
	    A number field $K$ is called a CM field if $K$ is totally imaginary, $F=K\cap\R$ is totally real, and $[K:F]=2$, i.e. $K$ is a totally complex quadratic extension of a totally real field.
	 \end{de}
CM fields can be characterised in terms of their units in the following way:
	 \begin{prop}[See \cite{remak} §2]\label{prop.rank.CMunits}
	    Let $K$ be a number field and $k$ a proper subfield. Let $r_K$ and $r_k$ be the ranks of the groups of units $\Oo_K^*$ and $\Oo_k^*$ respectively. Then $r_K=r_k$ if and only if $K$ is a CM field and $k=K\cap \R$.
	 \end{prop}

	 This Proposition implies that any number field $K$ which is not CM can be generated over $\Q$ by an element of $\Oo_K^*$.
	 The converse is not true in general: it can happen that a CM field is generated by a unit, as for example the fields $\Q(i)$ or $\Q(i\varphi)$ where $\varphi$ is the golden ratio. This is however ``uncommon'' as it can be shown (see for example \cite{remak}) that every totally real field $F$ has infinitely many CM extensions $K$ of degree 2, but only for finitely many of them $\Oo_F^*\subsetneq\Oo_K^*$.
	 
	 The general picture is given by the following theorem, whose proof is postponed to Section \ref{mainsection}.
	 \begin{thm}\label{thm:ClassGenUnit}
	  Let $K$ be a field. The following are equivalent:
	  \begin{enumerate}
	   \item\label{thm:enum:1} There does not exist $\xi\in \Oo_K^*$ such that $K=\Q(\xi)$,
	   \item\label{thm:enum:2} $K$ is a CM field and $\Oo_K^*\subset\R$,
	   \item\label{thm:enum:3} $K$ is a CM field and $\Oo_K^*=\Oo_F^*$.
	  \end{enumerate}
	 \end{thm}

   \subsection{Salem numbers}\label{subsec:salem}
Another class of algebraic numbers, closely related to Pisot numbers, is the class of Salem numbers.

      \begin{de}
	 A real algebraic integer  $\theta>1$ is called a Salem number if all its conjugates have modulus at most 1, and at least one conjugate lies on the unit circle.
      \end{de}     

Pisot numbers and Salem numbers are related by several topological properties. For example, any Pisot number is a limit of a sequence of Salem numbers; see for example \cite{PisotSalemNumbers}.

The requirement of one conjugate lying on the unit circle immediately leads to the following.
      \begin{rmk}\label{rem:salem}
         Let $\theta$ be a Salem number of degree $d$. Then $d-2$ of all conjugates of $\theta$ lie on the unit circle. Moreover, the minimal polynomial of $\theta$ is reciprocal, hence $\theta$ is a unit.
      \end{rmk}
     
The following result gives a simple description of the structure of fields generated by Salem numbers.
      \begin{prop}[Salem \cite{salem2} pages 163,169]\label{prop:salem}
         Let $K$ be a number field.

There exists a Salem number $\tau$ such that $K=\Q(\tau)$ if and only if $K$ has a totally real subfield $K'$ of index 2, and $K=K'(\tau)$ with $\tau+\tau^{-1}=\alpha$, where $\alpha>2$ is an algebraic integer in $K'$, all whose conjugates $\neq\alpha$ lie in $(-2,2)$. 
         
         If $K=\Q(\tau)$ for some Salem number $\tau$ of degree $n$, then there is a Salem number $\tau_0\in K$ such that the set of Salem numbers of degree $n$ in $K$ consists of the powers of $\tau_0$.
      \end{prop}

   \subsection{Height}\label{subsec:height}
   \subsubsection*{The logarithmic Weil height}
      The logarithmic Weil height is a function $h:\Qbar\to\R_{\geq 0}$ that measures the arithmetic complexity of an algebraic number. It is a powerful technical tool often used in number theory as a more advanced analogue of naive quantities such as the norm of the vector of the coefficients of the minimal polynomial of an algebraic number.
      
      We recall here that $h(x)=0$ if and only if $x$ is a root of unity or $x=0$, and that in any fixed number field $K$ and for any fixed $B\geq 0$ there are only finitely many elements of $K$ with $h(x)\leq B$.
      If $x=p/q\in\Q$ is a reduced fraction, then its height is $\log\max\{\abs{p},\abs{q}\}$.
      
      We will omit the full definition of the Weil height, as it is not needed for our purposes. We only give here the expression of the height of an algebraic integer, which has a simpler form.
      
      Let us write $\log^+(x)$ for $\max \{0,\log x\}$.
     Then for $\beta\in \Oo_K$
     \[h(\beta)=\frac{1}{[K:\Q]}\left(\sum_{i=1}^{r_1}\log^+\abs{\sigma_i(\beta)}+\sum_{i=1}^{r_2}2\log^+\abs{\tau_i(\beta)}\right).\]
      

If $\theta$ is a Pisot or Salem number (resp. complex Pisot number), the expression of the height can be further simplified, as only one summand is different from 0. In this case the height is given by $h(\theta)=\frac{\log\theta}{n}$ (resp. $h(\theta)=\frac{2\log\abs{\theta}}{n}$).

\subsubsection*{The Lehmer Conjecture}
      It is a theorem of Kronecker that the height of an algebraic number is zero if and only if it is 0 or a root of unity. The problem of determining how small the height can be is a fundamental one in Diophantine approximation.
      
      It follows from the elementary properties of the height, that this can be made arbitrarily small as the degree goes to infinity, for example $h(\sqrt[m]{2})=\frac{\log 2}{m}$.    
      
      It is conjectured (Lehmer's Conjecture) that \[h(\alpha)>\frac{c}{[\Q(\alpha):\Q]}\] for some positive constant $c$ when $\alpha$ is not $0$ or a root of unity, but this statement is still unproved.
      
      The smallest non-zero known value of the quantity $h(\alpha)[\Q(\alpha):\Q]$ is obtained when $\alpha$ is a root of
      \[x^{10}+x^9-x^7-x^6-x^5-x^4-x^3+x+1,\]
      which is a Salem number.
      
      The Lehmer conjecture, and stronger statements, have been proved under additional hypotheses on $\alpha$.
      For example if $\alpha$ (not $0$ or a root of unity) belongs to an abelian extension of $\Q$, then $h(\alpha)\geq\frac{\log 5}{12}$ (see \cite{AmDv}); if $\alpha$ belongs to a totally real field or a CM field, then $h(\alpha)\geq \frac{1}{2}\log\frac{1+\sqrt{5}}{2}$ (see \cite{Schinzel}). Smyth proved in \cite{Smyth} that if $\alpha$ has an odd degree, then $h(\alpha)\geq \frac{\log(\tau)}{[\Q(\alpha):\Q]}$ where $\tau= 1.3247\dots$ is the smallest Pisot number.
      
      The best explicit result which holds for any algebraic number is the following improvement of a theorem of Dobrowolski:
      \begin{thm}[\cite{Voutier}]\label{TeoDobro}
         Let $\alpha$ be a non-zero algebraic number of degree $n\geq 2$ which is not a root on unity. Then
         \[h(\alpha)>\frac{1}{4n}\left(\frac{\log\log n}{\log n}\right)^3.\]
      \end{thm}

\section{Locating Pisot units}\label{mainsection}
  It has been already stated in Subsection \ref{structure} that through the $\Log$ map, the group of units $\Oo^*$ is mapped to a lattice of rank $r=r_1+r_2-1$. In this section we will describe how Pisot and complex Pisot numbers behave under the $\Log$ map. 
  
   Let us define
\[Q_i=\{(x_1,\dotsc,x_{r+1})\in \R^{r+1} \mid x_i>0\text{ and }x_j<0\;\forall j\neq i\},\quad i=1,\dotsc, r+1\]
   and let $\overline{Q}_i$ be its topological closure.
   
  \begin{rmk}\label{rem:altezz}
    If $\beta\in \Oo_K\setminus\{0\}$ and $\Log(\beta)\in \overline{Q}_k$, then only one absolute value contributes to the height. In this case the height of $\beta$ is precisely the $k$-th component of $\Log(\beta)$.
  \end{rmk}
   Let us also define 
\[l_{k,j}=\left\{(x_1,\dotsc,x_{r+1})\in \R^{r+1} \mid x_k>0,\; x_j=-x_k,\; x_i=0\; \forall i\neq j,k\right\}\]
   for all pairs of distinct indices $k,j=1,\dotsc, r+1$. The $l_{k,j}$ are half-lines that lie in the boundary of the intersection $\overline{Q}_k$ and the hyperplane $\Hh$.
 
      It is clear from the definition that for a Pisot or a complex Pisot number $\beta\in K$ of degree $n$ all the entries of $\Log(\beta)$ are negative except the one corresponding to the identity embedding and therefore $\Log(\beta)\in Q_k$ where $\phi_k$ is the identity embedding. Similarly the image of a Salem number of degree $n$ under the $\Log$ map lies in the set $l_{k,j}$ for some $j$, where $\phi_k$ is the identity embedding.

Notice that it is necessary to assume in advance that $\beta$ has degree $n$ because Pisot numbers in $K$ of degree smaller than $n$ will be fixed by some embedding other than the identity, and therefore $\Log(\beta)$ will have more than just one positive coordinate.
   
   We will now show that this property characterises to some extent the (complex) Pisot numbers in $K$ of degree $n$:
   \begin{prop}\label{prop:main}
   Let $\beta\in \Oo_K\setminus\{0\}$.
      \begin{enumerate}
	 \item\label{GenReal} If $\Log(\beta)\in Q_k$ for some $k\leq r_1$, then $K=\Q(\beta)$;\\
	 if $\phi_k(\beta)>0,$ then $\phi_k(\beta)$ is a Pisot number.
	 \item\label{GenRealSalem} If $\Log(\beta)\in l_{k,j}$ for some $k\leq r_1$, then $K=\Q(\beta)$;\\
	 if $\phi_k(\beta)>0,$ then $\phi_k(\beta)$ is a Salem number.
	 \item\label{GenComplnCM} If $\Log(\beta)\in Q_k$ for some $k>r_1$ and $\phi_k(\beta)\not\in\R$, then  $K=\Q(\beta)$ and $\phi_k(\beta)$ is a complex Pisot number.
	 \item\label{GenComplCM} If $\Log(\beta)\in Q_k$ for some $k>r_1$ and $\phi_k(\beta)\in\R$, then $[K:\Q(\beta)]=2$;\\
	 if $\phi_k(\beta)>0$, then $\phi_k(\beta)$ is a Pisot number.
      \end{enumerate}
   \end{prop}
   \begin{proof}
      For every conjugate $\beta'$ of $\beta$ there are exactly $[K:\Q(\beta)]$ distinct embeddings $\phi:K\hookrightarrow\Qbar$ such that $\phi(\beta)=\beta'$.\\
      In cases \ref{GenReal} and \ref{GenRealSalem} we have that  $\phi_i(\beta)$ is the only conjugate of $\beta$ to lie outside of the unit circle because for all  $j\neq i$ we have that $\abs{\phi_j(\beta)}\leq 1$, therefore we must have that $[K:\Q(\beta)]=1$.
      
      In the remaining cases $\phi_i$ is a complex embedding and $\abs{\phi_i(\beta)}=\abs{\overline{\phi_i(\beta)}}>1$ while all other embeddings map $\beta$ inside the unit disk.\\
      In case \ref{GenComplnCM} we have that $\phi_i(\beta)\neq \overline{\phi_i(\beta)}$ because they are not real, and again $[K:\Q(\beta)]=1$.\\
      In case \ref{GenComplCM} we have that $\phi_i(\beta)= \overline{\phi_i(\beta)}$ and therefore $[K:\Q(\beta)]=2$.
   \end{proof}
   
   \begin{rmk}
    According to the Proposition, the elements of $\Lambda\cap l_{k,j}$ correspond to Salem numbers if $\phi_k$ is a real embedding. In fact Remark \ref{rem:salem} tells us that \[(\Lambda\setminus\{0\})\cap\partial\overline{Q}_k=\bigcup_{j\neq k}l_{k,j},\] and this equality holds (with the same proof) for all indices $k$, even those corresponding to complex embeddings.
   \end{rmk}

   What has been said until now applies to any algebraic integer $\beta$. We are especially interested in finding Pisot (or complex Pisot) generators which are units.
   A generic element of $\Oo^*$ can be written in a unique way as $u= \zeta u_1^{e_1}\dotsm u_r^{e_r}$ with $\zeta\in\mu_K$. If we write $e=(e_1,\dotsc,e_r)$ for the vector of exponents, then
      \begin{align}\label{equation:2}\Log(u)=Ae,\end{align}
where $A$ is the matrix from~\eqref{LabelMatrixA}.
   In light of this relation, the problem of deciding whether $\Log(u)\in Q_i$, or finding such a $u$ is reduced to a problem of linear algebra involving lattices and real matrices, assuming that a fundamental system has been provided.

   \subsection{The algorithm}
Now we can provide  algorithms to find a (complex) Pisot generator of a field $K$. Thanks to Proposition~\ref{prop:main}, we have to search for lattice points in a proper ``sector'' of $\R^{r+1}$, and we want them to minimise one of the coordinates.
\subsubsection*{Integer Linear Programming}
The problem of minimising a linear function on a domain defined by linear equalities and inequalities has been extensively studies since Dantzig's Simplex algorithm in 1947. When one impose the additional constraint that only solutions in integral numbers should be considered, the problem is known as Integer Linear Programming and there are several algorithms available. We mention for example the Cutting plane method developed by Gomory and Chvátal in the 50s, which, with various refinements, is used today in most applications.

It is outside the scope of this paper to describe fully the implementation of the integer programming algorithms, therefore we state precisely the version that we will use.
 \begin{algorithm}[H]
   \caption{{\tt INTPROG}}
   \label{AlgoIntProg}
  \begin{algorithmic}[1]
    \REQUIRE A matrix $A\in\R^{m\times n}$, two vectors $b\in\R^m,c\in\R^n$, a real $\epsilon\geq 0$.
    \ENSURE A list of vectors $[x_0,\dotsc,x_s]\in\Z^n$ such that:\\
    $A^i \cdotp x_j\leq b_i$ holds for all indices $j=0,\dotsc,s$ and $i=1,\dotsc,m$,\\
    $x_0$ minimizes the quantity $c\cdotp x$ under the constraints above,\\
    $c\cdotp x_0\leq c\cdotp x_1\leq\dotsb \leq c\cdotp x_s\leq (1+\epsilon)c\cdotp x_0$,\\
    $\{x_0,x_1,\dotsc,x_s\}$ is the set of all vectors $x\in\Z^n$ such that $c\cdotp x\leq (1+\epsilon)c\cdotp x_0$.
  \end{algorithmic}
 \end{algorithm}
 A description of a possible implementation is found in \cite{MLIP}.
 
 We remark that standard Integer Programming algorithms usually find only a single optimal solution. This is enough for our Algorithms \ref{Algo1} and \ref{Algo2}, but for Algorithm \ref{Algo3} we need this improved version which enumerates all quasi-optimal solutions. This should be taken into account when implementing the algorithms, as it could lead to a difference in performance.
 
\subsubsection*{Algorithms for real fields}
The first algorithm that we present uses Integer Programming to find the point in $\overline{Q}_i$ which corresponds to a generating unit of smallest height.
We use Theorem \ref{TeoDobro} to ensure that the algorithm finds a non-torsion unit.
We define 
\[
\delta(n)=\begin{cases}
   \frac{1}{4n}\left(\frac{\log\log n}{\log n}\right)^3 &\text{if $n\geq 4$ and even,}\\
   \log(1.32) &\text{otherwise.}
          \end{cases}
\]

 \begin{algorithm}[H]
   \caption{\tt FINDMIN}
   \label{Algo1}
  \begin{algorithmic}[1]
    \REQUIRE A fundamental system of units $(u_1,\dots,u_r)$ of $K$ and an index $k=1,\dotsc,r+1$.
    \ENSURE A unit $\alpha\in\Oo^*$ such that $\Log(\alpha)\in \overline{Q}_k\setminus\{0\}$ and $h(\alpha)$ is minimal among such units.
    \STATE Construct the matrix $A$ as in \eqref{LabelMatrixA}
    \STATE Construct the matrix $B$ obtained from $A$ by changing the sign of each element of the $k$-th row.
    \STATE Set $c\in\R^r$ equal to the $k$-th row of $A$
    \STATE Set $b$ equal to  the vector $(b_1,\dotsc,b_{r+1})$ with $b_i=-\delta(n)$ if $i=k$ and $b_i=0$ otherwise.
    \STATE Call {\tt INTPROG} with input $(B,b,c,0)$. Let $e\in\Z^r$ be the first element of the output.
    \RETURN $\alpha=\prod_{i=1}^r u_i^{e_i}$.
  \end{algorithmic}
 \end{algorithm}

\begin{proof}[Proof of Theorem \ref{thm:algo1}]
    The set $Q_k\cap\Lambda$ is non-empty because $\Q\otimes\Lambda$ is dense in $\Hh$ and $\Hh\cap Q_i$ is an open subset of $\Hh$ which is invariant under multiplication by a positive scalar.
   Therefore integer programming will always terminate and will find a point.
   If $K$ is a real field, and $\phi_k$ is the identity embedding, then by Proposition \ref{prop:main} $\abs{\alpha}$ has the required properties.
   
   Notice that the constraint given by $b_i=-\delta(n)$ excludes the origin but no other point of the lattice.
\end{proof}

Notice that when $K$ is real and $\phi_k$ the identity the output of {\tt FINDMIN} does not necessarily yield a Pisot generator of $K$ since it might happen that $\Log(\alpha)\in l_{k,j}.$ In this case $K$ is generated by the Salem number $\abs{\alpha}$, and any Salem number in $K$ can be written as $\alpha^n$, which is a direct consequence of Proposition~\ref{prop:salem} and the minimality of $\alpha$. 

 When this happens, we can exclude all the Salem numbers whose $\Log$ lies in $\overline{Q}_k$  by changing the constraints in {\tt INTPROG}. The following theorem shows how.
\begin{thm}\label{thm:cutedge}
   Let $K$ be a number field, $k\in\{1,\dotsc,r+1\}$, and let $\alpha\in \Oo_K^*$ be the output of Algorithm {\tt FINDMIN}. Assume that $\Log(\alpha)\in l_{k,j}$ for some index $j$.
   Then every unit $\beta$ such that \[\Log(\beta)\in\overline{Q}_k\cap \big\{(x_1,\dotsc,x_{r+1})\in \R^{r+1} \mid\sum_{i\neq k,j} x_i>-\delta(n)\big\}\] is a power of $\alpha$ times a root of unity.   
\end{thm}
\begin{proof}
   Let $\beta$ be a unit such that $\Log(\beta)=(b_1,\dotsc,b_{r+1})\in\overline{Q}_k$ and $-\delta(n)< \sum_{i\neq k,j} b_i\leq 0$. Dividing $\beta$ by a suitable power of $\alpha$ we can also assume that $-h(\alpha)\leq b_j< 0$.
   
   
   We have that
   \[h(\beta/\alpha)=\sum_{i\neq j,k} \max\{0,b_i\}+\max\{0,b_j+h(\alpha)\}+\max\{0,-\sum_{i\neq k}b_i-h(\alpha)\}.\]
   But $b_i\leq 0$ for all $i\neq k$ because $\Log(\beta)\in \overline{Q}_k$, $b_j+h(\alpha)\geq 0$ by our assumption on $\beta$, and $-\sum_{i\neq k}b_i=b_k=h(\beta)\geq h(\alpha)$ because $\alpha$ is the output of Algorithm {\tt FINDMIN}.
   Therefore
   \[h(\beta/\alpha)=b_j+h(\alpha)-\sum_{i\neq k}b_i-h(\alpha)=-\sum_{i\neq j,k} b_i<\delta(n).\]
   Therefore by Theorem \ref{TeoDobro} the ratio $\beta/\alpha$ is a root of unity.
\end{proof}

We now give an algorithm implementing the previous Theorem. 
 
\begin{algorithm}[H]
  \caption{{\tt CUTEDGE}}
  \label{Algo2}
  \begin{algorithmic}[1]
    \REQUIRE A fundamental system of units $(u_1,\dots,u_r)$ of $K$ and an index $k=1,\dotsc,r+1$
    \ENSURE A unit $\beta\in\Oo_K^*$ such that $\Log(\beta)\in Q_k$ and $h(\beta)$ is minimal among such units
    \STATE Run Algorithm 1 with the same input. Let $\alpha$ be the output
    \IF{$\Log(\alpha)\in Q_k$}
      \RETURN $\alpha$
    \ENDIF    
    \STATE Let $j$ be the index such that $\Log(\alpha)\in l_{k,j}$
    \STATE Construct the matrix $A$ as in Algorithm 1
    \STATE Construct the matrix $B'\in\R^{(r+1)\times r}$ obtained from $A$ by replacing the $k$-th row with $-A^k-A^j$
    \STATE Set $c\in\R^r$ equal to the $k$-th row of $A$ as in Algorithm 1
    \STATE Set $b\in\R^{r+1}$ equal to  the vector $(b_1,\dotsc,b_{r+1})$ with $b_k=-\delta(n)$ and $b_i=0$ for $i\neq k$
    \STATE Call algorithm {\tt INTPROG} with input $(B',b,c,0)$. Let $e'\in\Z^r$ be the first element of the output
    \RETURN $\beta=\prod_{i=1}^r u_i^{e'_i}$
   \end{algorithmic}
 \end{algorithm}

 \begin{proof}[Proof of Theorem \ref{thm:algo2}]
 Arguing as in the proof of Theorem \ref{thm:algo1} we see that the feasible region delimited by the constraints of the Integer Programming always contains integral points and therefore {\tt INTPROG} will return one of them.
 
 If the $\alpha$ obtained in step 1 is a Pisot number, than the conclusion of the Theorem is clear. If it is a Salem number, then by Theorem \ref{thm:cutedge} the final output of the algorithm must be a Pisot number because the constraint added in step 7 excludes all the Salem numbers whose $\Log$ lies in $\overline{Q}_k$.
 Then $\abs{\alpha}$ has all the required properties again by Proposition \ref{prop:main}.
\end{proof}

\bigskip

Here is an example of a field containing both Salem and Pisot numbers.
Let $\alpha=100.960\dots$ be the bigger real root of $x^4-101x^3+5x^2-101x+1$; this is a Salem number. When running algorithm {\tt FINDMIN} on the field $K=\Q(\alpha)$ the output is $\alpha$ itself, which is the smallest U-number in $\Q(\alpha)$. When running algorithm {\tt CUTEDGE} the output is the biggest real root of 
\begin{equation*}
   x^4+a_3x^3+a_2x^2+a_1x+1,
\end{equation*}
where
\begin{align*}
   a_3&=-60048257490013814123246164511189751124091132508231119928295605154893060\\
   a_2&=2556025223049864739934292009524109324899782644859727711036043393301222\\
   a_1&=97520402335817024268676911493103325.
\end{align*}

This is approximately $6\cdot 10^{70}.$ This example shows that there can be relatively a big gap between the smallest Salem and Pisot generators of the same field.

\bigskip

  This algorithm solves our motivating problem when the field is real, but several problems might occur when the field $K$ is complex. Firstly, after finding an element in $\Lambda\cap Q_k$, we cannot conclude that it corresponds to a generating unit, because we could be in case \ref{GenComplCM} of Proposition \ref{prop:main}. Furthermore, the solution to the Integer Programming problem might not be unique because there could be non-torsion units of modulus 1 (in fact, this always happens unless $K\cap \overline{K}$ is real or CM, see \cite{daileda}). It is to overcome this problems that we use the algorithm {\tt INTPROG} as stated in Algorithm \ref{AlgoIntProg}. We will in fact enumerate all points of $\Lambda$ in a thin slice of $Q_k$ and look for the smallest that meets our requirements.

 \subsubsection*{Algorithm for complex fields}

 \begin{algorithm}[H]
  \caption{FINDCPISOT}
  \label{Algo3}
  \begin{algorithmic}[1]
    \REQUIRE A fundamental system of units $(u_1,\dots,u_r)$ of a complex number field $K$, a generator $\zeta$ of $\mu_K$  and the index $k$ of the identity embedding.
    \ENSURE A unit $\beta\in\Oo_K^*$ such that $K=\Q(\beta)$, $\Log(\beta)\in Q_k$ and $h(\beta)$ is minimal among such units, if any exists. An error message otherwise.
    \IF{$\zeta\neq-1$}
      \STATE Run Algorithm {\tt CUTEDGE} with input $((u_1,\dotsc,u_r),k)$. Let $\beta$ be its output.
      \IF{$\beta\not\in\R$}
	\RETURN $\beta$
      \ELSE
	\RETURN $\zeta\beta$
      \ENDIF
    \ENDIF
    \IF{$u_i\in\R$ for all $i$}
      \RETURN \FALSE
    \ELSE
      \STATE Construct the matrix $A$ as in \eqref{LabelMatrixA}
      \STATE Construct the matrix $B$ obtained from $A$ by changing the sign of each element of the $k$-th row.
      \STATE Set $c\in\R^r$ equal to the $k$-th row of $A$
      \STATE Set $b$ equal to  the vector $(b_1,\dotsc,b_{r+1})$ with $b_i=-\delta(n)$ if $i=k$ and $b_i=0$ otherwise.
      \WHILE{\TRUE}\label{algo3:whileloop}
	\STATE Call {\tt INTPROG} with input $(B,b,c,10^{-3})$. Let $E=[e^{(0)},\dotsc,e^{(s)}]$ be the output.
	\FOR{$j=0$ to $s$}
	  \STATE Let $\beta=\prod_{i=1}^r u_i^{e^{(j)}_i}$
	  \IF{$\beta\not\in\R$ \AND $\Log(\beta)\in Q_k$}
	    \RETURN $\beta$
	  \ENDIF
	\ENDFOR
	\STATE $b_k\longleftarrow -(1+10^{-3})c\cdotp e^{(0)}$
      \ENDWHILE
    \ENDIF
   \end{algorithmic}
 \end{algorithm}

 \begin{proof}[Proof of Theorem \ref{thm:algo3}]
   If there are non-trivial roots of unity, then the Algorithm terminates because of Theorem \ref{thm:cutedge} and the output has the desired properties because of Proposition \ref{prop:main}.
   
   If $\mu_K=\{\pm 1\}$ and all the $u_i$ are real, then the Algorithm stops with an error message. In this case clearly no generating unit can exist.
   
   Let $u_j$ be a complex fundamental unit. Arguing as in the proof of Theorem \ref{thm:algo1} we can see that {\tt INTPROG} will always find a non-torsion unit $\beta$ such that $\Log(\beta)\in Q_k$. If $\beta\not\in\R$ then this will be the output of FINDCPISOT; if $\beta\in\R$ then $\Log(u_j \beta^m)\in Q_k$ for any sufficiently big $m$, and we can be sure that the while-loop in step \ref{algo3:whileloop} will be escaped. The rest of the thesis follows again by Proposition \ref{prop:main}.
  \end{proof}
  We will now give the proof of Theorem \ref{thm:ClassGenUnit} in Section \ref{preliminaries}.
   \begin{proof}[Proof of Theorem \ref{thm:ClassGenUnit}]
      Statements \ref{thm:enum:2}. and \ref{thm:enum:3}. are clearly equivalent, and they immediately imply \ref{thm:enum:1}.
      Let us now prove first that statement \ref{thm:enum:1} implies that $K$ is a CM field.
      
      Assume that $K$ is not CM and let $K_1,\dotsc,K_m$ be the proper subfields of $K$.  By Proposition \ref{prop.rank.CMunits} the $\Oo_{K_i}^*$ are finitely many submodules of $\Oo_K^*$ of strictly smaller rank, therefore $\displaystyle\cup_{i=1}^m \Oo_{K_i}^*\subsetneq \Oo_K^*$, and there exist some units of $K$ not lying in any proper subfield, which contradicts \ref{thm:enum:1}.
      
      To finish the proof we observe that if $\Oo_K^*\nsubseteq\R$ then by Theorem \ref{thm:algo3} $K$ is generated by a unit.
  \end{proof}
  
  \subsection{Height bound}\label{subsec:bound}
   We conclude the paper with the proof of a bound for the height of the smallest Pisot unit of degree $n$ in a real field $K$.

  \begin{thm}
    Let $1\leq i \leq r$. There exist a non-zero vector $e\in\Z^r$ such that
    \begin{align*}
     &A^j\cdotp  e < 0 &\forall j\neq i,\\
     &A^i\cdotp e\leq\frac{\sqrt{r}}{2}\sum_{j=1}^{r+1}\norm{A^j}.
    \end{align*}
    In other words, for any $i$ there exists an element $\beta\in \Oo_K^*$ such that $\Log(\beta)\in Q_i$ and $h(\beta)\leq \frac{\sqrt{r}}{2}\sum_{j=1}^{r+1}\norm{A^j}$. If $K$ is a real field and $\phi_i$ is the identity embedding, then $\abs{\beta}$ is a Pisot unit of degree $n$.
  \end{thm}
  \begin{proof}
    Let $\epsilon>0$ and let $w$ be the solution to the linear system given by
  \begin{align*}
    A^j \cdotp w = -\left(\frac{\sqrt{r}}{2}+\epsilon\right)\norm{A^j},\quad \forall j\neq i.
  \end{align*}
  A unique solution exists because the matrix of this linear system has nonzero determinant (in fact, the absolute value of the determinant is equal to $\Reg(\Oo_K)$).
  
   Inside the ball of radius $\sqrt{r}/2$ centred in $w$ there clearly exists a vector $e$ with integral coordinates. For every $j\neq i$ we have
  \begin{align*}
    A^j\cdotp e&=A^j\cdotp w +A^j\cdotp(e-w)\leq  -\left(\frac{\sqrt{r}}{2}+\epsilon\right)\norm{A^j}+\abs{A^j\cdotp (e-w)}\leq\\
    &\leq-\left(\frac{\sqrt{r}}{2}+\epsilon\right)\norm{A^j}+\norm{A^j}\norm{e-w}\leq\\
    &\leq-\left(\frac{\sqrt{r}}{2}+\epsilon\right)\norm{A^j}+\frac{\sqrt{r}}{2}\norm{A^j}= -\epsilon\norm{A^j}<0.
   \end{align*}
  Furthermore,
  \begin{multline*}
    A^i\cdotp e=A^i\cdotp w+A^i\cdotp (e-w)\leq -\sum_{j\neq i}A^j\cdotp w+\abs{A^i\cdotp (e-w)}\leq \\
    \leq\left(\frac{\sqrt{r}}{2}+\epsilon\right)\sum_{j\neq i}\norm{A^j}+\norm{A^i}\norm{e-w}\leq \left(\frac{\sqrt{r}}{2}+\epsilon\right)\sum_{j\neq i}\norm{A^j}+\frac{\sqrt{r}}{2}\norm{A^i}=\\
    =\frac{\sqrt{r}}{2}\sum_{j=1}^{r+1}\norm{A^j}+\epsilon\sum_{j\neq i}\norm{A^j}.
  \end{multline*}
  We now let $\epsilon$ go to zero, and the bounds in the statement follow because the lattice $\Z^r$ is discrete.
  
  To prove the final statements we see that $\beta=u_1^{e_1}\dotsb u_r^{e_r}$ has the desired properties because of \eqref{equation:2}, Remark \ref{rem:altezz} and Proposition \ref{prop:main}.
\end{proof}



 \section*{Acknowledgements}
We thank Professor Robert Tichy and the TU Graz for hosting and supporting us in the initial stage of this work.
This work has been supported by the Czech Science Foundation grant GA\v CR 13-03538S and by the Grant Agency of the Czech Technical University in Prague, grant No. SGS11/162/OHK4/3T/14.

\nocite{PisotSalemNumbers}
\bibliography{vave-biblio}{}
\bibliographystyle{alpha}

\medskip

\noindent Francesco Veneziano:\\
Mathematisches Institut,
Universit\"{a}t Basel,
Spiegelgasse 1,
CH-4051 Basel,
Switzerland.\\
email: francesco.veneziano@unibas.ch\\

\noindent Tomáš Vávra:\\
Departments of Mathematics,
Faculty of Nuclear Sciences and Physical Engineering, CTU in Prague,
Trojanova 13,
120 00 Prague 2,
Czech Republic.\\
email: t.vavra@seznam.cz

\end{document}